\documentclass[12pt, a4paper, reqno]{amsart}
\usepackage{amsmath}
\usepackage{amsfonts}
\usepackage{amssymb}
\usepackage{amsthm}
\usepackage{url}
\usepackage{graphicx}

\newcommand{\R}{\mathbb{R}}

\def\vint_#1{\mathchoice%
          {\mathop{\kern 0.2em\vrule width 0.6em height 0.69678ex depth -0.58065ex
                  \kern -0.8em \intop}\nolimits_{\kern -0.4em#1}}%
          {\mathop{\kern 0.1em\vrule width 0.5em height 0.69678ex depth -0.60387ex
                  \kern -0.6em \intop}\nolimits_{#1}}%
          {\mathop{\kern 0.1em\vrule width 0.5em height 0.69678ex depth -0.60387ex
                  \kern -0.6em \intop}\nolimits_{#1}}%
          {\mathop{\kern 0.1em\vrule width 0.5em height 0.69678ex depth -0.60387ex
                  \kern -0.6em \intop}\nolimits_{#1}}}

\newcommand{\art}[6]{{\sc #1, \rm #2, \it #3 \bf #4 \rm (#5), \mbox{#6}.}}
\newcommand{\book}[3]{{\sc #1, \it #2, \rm #3.}}
\newcommand{\AND}{{\rm and }}
\newcommand{\p}{{$p\mspace{1mu}$}}

\newcommand{\Om}{\Omega}
\newcommand{\loc}{_{\rm loc}}

\newcommand{\eps}{\varepsilon}

\DeclareMathOperator*{\ext}{ext}
\DeclareMathOperator*{\inte}{int}

\theoremstyle{plain}
\newtheorem{theorem}[equation]{Theorem}
\newtheorem{lemma}[equation]{Lemma}

\numberwithin{equation}{section}

\theoremstyle{definition}

\theoremstyle{remark}
\newtheorem{remark}[equation]{Remark}
\title[Some remarks on sign changing solutions]{Some remarks on sign
  changing solutions of a quasilinear elliptic equation in two
  variables}

\author{Seppo Granlund} \address[S.G.]{University of Helsinki,
  Department of Mathematics and Statistics, P.O. Box 68, FI-00014
  University of Helsinki, Finland} \email{seppo.granlund@pp.inet.fi}

\author{Niko Marola} \address[N.M.]{University of Helsinki, Department
  of Mathematics and Statistics, P.O. Box 68, FI-00014 University of
  Helsinki, Finland} \email{niko.marola@helsinki.fi}

\begin{document}



\begin{abstract}
  We consider planar solutions to certain quasilinear elliptic
  equations subject to the Dirichlet boundary conditions; the boundary
  data is assumed to have finite number of relative maximum and
  minimum values. We are interested in certain vanishing properties of
  sign changing solutions to such a Dirichlet problem.
  Our method is applicable in the plane.
\end{abstract}

\maketitle

\section{Introduction and preliminaries}

In this paper we consider solutions of quasilinear second order
elliptic partial differential equations of the form
\begin{equation} \label{eq} \nabla\cdot \mathcal{A}(x,\nabla u)=
  \mathcal{B}(x,\nabla u),
\end{equation}
where $\mathcal{A}\colon\R^2\times\R^2 \to\R^2$ and
$\mathcal{B}\colon\R^2\times\R^2 \to\R$ are Carath\'eodory functions
under certain structural conditions discussed in
Section~\ref{sect:structural}. A noteworthy example of such equations
is the \p-Laplace equation
\begin{equation} \label{eq:pLap} 
\nabla\cdot(|\nabla u|^{p-2}\nabla u)=0,
\end{equation}
where $1<p<\infty$, which gives the Laplace equation when $p=2$; we
refer to \cite{Lindqvist}.

The result of this note is the following. Let $G$ be a bounded
simply-connected Jordan domain in $\R^2$. Suppose $u$ is a solution to
\eqref{eq} subject to the Dirichlet boundary condition
\[
u=g\,\textrm{ on }\,\partial G,
\]
where $g\in W^{1,p}(G)\cap C(\overline{G})$. We assume minimal
regularity conditions on $\partial G$ so that every boundary point is
regular, and hence $u\in C(\overline{G})$. If $g|_{\partial G}$ has
finite number of relative maximum and minimum values, and we assume
further that for all $x\in G$ there exists $r_x>0$ such that for all
$r\leq r_x$ the set $\{z\in B_r(x)\subset G:\, u(z)=0\}$ is connected,
then
if $u$ vanishes in some open subset of $G$, then $u$ vanishes
identically in $G$.

For linear equations the study of certain vanishing properties, unique
continuation in particular, is rather complete \cite{GarLinIndiana,
  GarLinCPAM}. The non-linear case, on the other hand, is more or less
open, although there are some results; we refer to
\cite{Alessandrini87}, \cite{Alessandrini}, \cite{ALR},
\cite{AlesSiga}, \cite{ArmSil}, \cite{BoIw}, \cite{Rossi},
\cite{Manfredi}, and \cite{Martio}. In \cite{Martio}, in particular,
some counterexamples are constructed in the case of $p=n$ in
\eqref{eq:structural}, $n\geq 3$, and $\mathcal{B}\equiv 0$. More
precisely, the author shows that there is a solution that vanishes in
the lower half space $x_n < 0$ of $\R^n$ but does not vanish
identically in $\R^n$.

Our approach in this note is based on the analysis of
nodal domains, which are maximal connected components of the set
\[
\{x\in G: u(x)\neq 0\},
\]
and nodal lines 
\[
\overline{\{x\in G: u(x)=0\}},
\]
which are the boundaries of nodal domains. Our main tool is to couple
the strong maximum principle and the Harnack inequality with some
topological arguments; this argument applies in the situation in which
there are finite number of nodal domains. 

The topological approach taken in the paper can be applied also to the
nonlinear eigenvalue problem involving the \p-Laplacian and to more
general eigenvalue problems constituting the Fu\v cik spectrum. We
refer to a recent paper \cite{GraMar3} for more details.

Finally, we want to refer to \cite{AlessandriniPisa} since the framework and
some ideas there are somewhat related to those taken in the present
note. 
%

\subsection{Notation} Throughout, $G$ is a bounded simply-connected
Jordan domain of $\R^2$. A domain is an open connected set in
$\R^2$. We write $B_r = B_r(x) = B(x,r)$ for concentric open balls of
radii $r>0$ centered at some $x\in G$. We denote the closure,
interior, exterior, and boundary of $E$ by $\overline{E}$, $\inte(E)$,
$\ext(E)$, and $\partial E$, respectively.

\subsection{Structural assumptions}
\label{sect:structural}

Let us specify the structure of $\mathcal{A}$ and $\mathcal{B}$ in
\eqref{eq}; We shall assume that there are constants $0<a_0\leq
a_1<\infty$ and $0<b_1<\infty$ such that for all vectors $h$ in $\R^2$
and almost every $x\in G$ the following structural assumptions apply
\begin{equation} \label{eq:structural} \left. \begin{array}{l}
      \mathcal{A}(x,h)\cdot h \geq a_0|h|^p,  \\
      |\mathcal{A}(x,h)|\leq a_1|h|^{p-1}, \\
      |\mathcal{B}(x,h)| \leq b_1|h|^{p-1}
    \end{array} \right\}
\end{equation}
where $1<p<\infty$. We do not assume the monotoneity or the
homogeneity of the operator $\mathcal{A}$ since we do not consider
existence or uniqueness problems.

The structural conditions \eqref{eq:structural} result in H\"older
continuity of a weak solution to \eqref{eq}, and moreover in the
Harnack inequality and the strong maximum principle, we refer to
Serrin~\cite{Serrin}.

We could also allow for the following structural conditions
\begin{equation} \label{eq:structural2} \left. \begin{array}{l}
      \mathcal{A}(x,h)\cdot h \geq a_0|h|^p,  \\
      |\mathcal{A}(x,h)|\leq a_1|h|^{p-1}, \\
      |\mathcal{B}(x,h)| \leq b_0|h|^p + b_1|h|^{p-1}
    \end{array} \right\}
\end{equation}
where $0<b_0<\infty$ and $1<p<\infty$. In this case local H\"older
continuity and the Harnack inequality for a locally bounded weak
solution of \eqref{eq} follow from Trudinger~\cite{Trudinger}.

We do not consider the case in which $\mathcal{A}$ and $\mathcal{B}$
may depend on $u$, or, for that matter, aim at the most general
structure in \eqref{eq:structural} or \eqref{eq:structural2}. We will
only need that solutions of \eqref{eq} are continuous, and satisfy the
Harnack inequality and the strong maximum principle.
%

\subsection{Some elements of the plane topology}
\label{sect:topology}

We recall a few facts about the topology of planar sets; a good
reference is \cite{Newman}. Let $\Omega$ be any domain in $\R^2$. A
Jordan arc is a point set which is homeomorphic with $[0,1]$, wheras a
Jordan curve is a point set which is homeomorphic with a circle. By
Jordan's curve theorem a Jordan curve in $\R^2$ has two complementary
domains, and the curve is the boundary of each component. One of these
two domains is bounded and this domain is called the interior of the
Jordan curve. A domain whose boundary is a Jordan curve is called a
Jordan domain.

As a related note, it is well known that the boundary of a bounded
simply-connected domain in the plane is connected. In the plane a
simply-connected domain $\Om$ can be defined by the property that all
points in the interior of any Jordan curve, which consists of points
of $\Om$, are also points of $\Om$ \cite{Nehari}.

A Jordan arc with one end-point on $\partial\Omega$ and all its other
points in $\Omega$, is called an end-cut. If both end-points are in
$\partial \Omega$, and the rest in $\Omega$, a Jordan arc is said to
be a cross-cut in $\Omega$. A point $x\in\partial\Omega$ is said to be
accessible from $\Omega$ if it is an end-point of an end-cut in
$\Omega$. Accessible boundary points of a planar domain are aplenty:
The accessible points of $\partial \Omega$ are dense in
$\partial\Omega$ \cite[p. 162]{Newman}.

We recall a few facts about connected sets and $\eps$-chains. If $x$
and $y$ are distinct points, then an $\eps$-chain of points joining
$x$ and $y$ is a finite sequence of points
\[
x=a_1,\, a_2,\,\ldots,\,a_k=y
\]
such that $|a_i-a_{i+1}|\leq\eps$, for $i=1,\,\ldots,\,k-1$. A set of
points is $\eps$-connected if every pair of points in it can be joined
by an $\eps$-chain of points in the set. A compact set $F$ in $\R^2$
is connected if and only if it is $\eps$-connected for every $\eps>0$
\cite[Theorem~5.1, p. 81]{Newman}. If a connected set of points in
$\R^2$ intersects both $\Omega$ and $\R^2\setminus\Omega$ it
intersects $\partial \Omega$ \cite[Theorem~1.3, p. 73]{Newman}.

Lastly, we recall the following topological property
\cite[p. 159]{Newman}. A subset $E$ of $\R^2$ is said to be locally
connected at any $x\in\R^2$ if for every $\eps>0$ there exists
$\delta>0$ such that any two points of $B_\delta(x)\cap E$ are joined
by a connected set in $B_\eps(x)\cap E$. A set is uniformly locally
connected, if for every $\eps>0$ there exists $\delta>0$ such that all
pairs of points, $x$ and $y$, for which $|x-y|<\delta$ can be joined
by a connected subset of diameter less than $\eps$. All convex domains
and, more generally, Jordan domains are uniformly locally connected
\cite[Theorem 14.1, p. 161]{Newman}. However, simply-connected domains
are not necessarily locally connected.

\section{Vanishing properties and nodal domains}

We may interpret equation \eqref{eq} in the weak sense. We recall that
it follows from the structural assumptions \eqref{eq:structural} (or
\eqref{eq:structural2}) that a weak solution to \eqref{eq} is H\"older
continuous and satisfies the following Harnack inequality. We refer to
Serrin~\cite{Serrin}. The proof is based on the Moser iteration method
\cite{Moser}.

\begin{theorem}[Harnack's inequality] \label{thm:Harnack} Suppose $u$
  is a non-negative solution to \eqref{eq} in $B_{3r}\subset G$. Then
\[
\sup_{B_r} u\leq C\inf_{B_r} u,
\]
where $C=C(p,\,a_0,\,a_1,\,b_1)$.
\end{theorem}

Having the structure \eqref{eq:structural2} in \eqref{eq},
Theorem~\ref{thm:Harnack} can be found in
Trudinger~\cite{Trudinger}. Moreover, in this case we shall assume
that a weak solution $u$ is locally bounded.

We also point out the following important property, the strong maximum
principle, which can be deduced from the Harnack inequality. We refer
to a monograph by Pucci and Serrin~\cite{PucciSerrin} on maximum
principles.

\begin{theorem}[Strong maximum principle]
  Suppose $u$ is a non-constant solution to \eqref{eq} in $G$. Then
  $u$ cannot attain its maximum at an interior point of $G$.
\end{theorem}

We shall make use of the fact that if $u$ is a solution to \eqref{eq},
then $-u+c$, $c\in\R$, is also a solution to an equation similar to
\eqref{eq}. Hence Harnack's inequality and the strong maximum
principle apply to both $u$ and $-u+c$.

Our main result is the following theorem. We assume minimal regularity
conditions on $\partial G$ so that every boundary point is regular,
and hence $u\in C(\overline{G})$.

\begin{theorem} \label{thm:uniquecont} Suppose $u$ is a solution to
  the equation \eqref{eq} under structural conditions
  \eqref{eq:structural} (or \eqref{eq:structural2}, and $u$ locally
  bounded) in a bounded simply-connected Jordan domain $G$ of $\R^2$
  subject to the Dirichlet condition 
  \[
  u=g \quad \textrm{on} \quad \partial G, 
  \]
  where $g\in W^{1,p}(G)\cap C(\overline{G})$. We assume further that
  for all $x\in G$ there exists $r_x>0$ such that for all $r\leq r_x$
  the set $\{z\in B_r(x)\subset G:\, u(z)=0\}$ is connected. If
  $g|_{\partial G}$ has finite number of relative maximum and minimum
  values, and
  if $u$ vanishes in some open subset of $G$, then $u$ vanishes
  identically in $G$.
\end{theorem}

We could also state the result as follows: If $u$ is a constant in
some open subset of $G$, then $u$ is identically constant in $G$. In
what follows, however, we stick to the classical formulation by
dealing with a vanishing solution.

The crux of the proof is to study so-called nodal domains. A maximal
connected component, i.e. one that is not a strict subset of any other
connected set, of the set $\{x\in G: u(x)\neq 0\}$ is called, in what
follows, a nodal domain. We denote these components by
\[
N_i^+=\{x\in G: u(x)>0\}, \quad \textrm{and} \quad N_j^-=\{x\in G:
u(x)<0\},
\]
where $i,j=1,2,\ldots\,$. We remark that if $u$ is, for instance, a
solution to the \p-Laplace equation \eqref{eq:pLap} it is not known
whether the number of nodal domains is finite.

In the proof of the following key lemma, we make use of the fact that
$G$ is a Jordan domain, and more precisely, $G$ is uniformly locally
connected at every $x\in \partial G$, we refer to
Section~\ref{sect:topology}.

\begin{lemma} \label{lemma:numbernodal} Let the hypothesis of
  Theorem~\ref{thm:uniquecont} be satisfied. Then the number of nodal
  domains, $N_i^+$ and $N_j^-$, is finite.
\end{lemma}

\begin{remark}
  The extra assumption, for all $x\in G$ there exists $r_x>0$ such
  that for all $r\leq r_x$ the set $\{z\in B_r(x)\subset G:\,
  u(z)=0\}$ is connected, in Theorem~\ref{thm:uniquecont} can be
  omitted in Lemma~\ref{lemma:numbernodal}.
\end{remark}

\begin{proof}[Proof of Lemma~\ref{lemma:numbernodal}]
  We note first that $u$ vanishes on all nodal lines in $G$, i.e. on
  $\partial N_i^+\cap G$ and $\partial N_j^-\cap G$. Hence by the
  strong maximum principle each nodal line meets the boundary of $G$.

  By the strong maximum principle the set $\partial N_i^+\cap\partial
  G$ contains a global maximum of $u$ in $N_i^+$. We then show that
  such maximum point $x_0\in \partial N_i^+$ is also a relative
  maximum of $g$ on $\partial G$: Let $x_0\in \partial G$ be a maximum
  point of $u$ on some fixed nodal domain $N_i^+$. We shall then apply
  local connedtedness of $G$ at every boundary point $x\in\partial G$
  (Section~\ref{sect:topology}). We claim next that there exists
  $\delta_{x_0}>0$ such that $B_\delta(x_0)\cap G$, for each
  $\delta<\delta_{x_0}$, contains only points of $N_i^+$. But assume,
  for now, that this is not the case. Hence for each
  $\delta<\delta_{x_0}$ there exists $\tilde x\in B_\delta(x_0)\cap G$
  such that $\tilde x$ belongs to some other nodal domain than
  $N_i^+$, say, $N_j^-$ or $u(\tilde x)=0$. Obviously, we need to
  consider only the former case.

  We write $u(x_0)=\max_{x\in N_i^+}u(x)=\sigma>0$. There exists a
  positive $\tilde\delta<\delta_{x_0}$ such that
  \[
  u(x) > \frac{\sigma}{2}
\]
for every $x\in \partial N_i^+\cap \partial G\cap
B_{\tilde\delta}(x_0)$ (it can be verified that such points exist
since $G$ is a Jordan domain). Moreover, since there exists a point
$\tilde x \in B_{\tilde\delta}(x_0)\cap G$ such that $\tilde x\in
N_j^-$ and $G$ is locally connected at every $x\in \partial G$, there
must exist also a point $\bar x\in B_{\tilde\delta}(x_0)\cap G$ so
that $u(\bar x)=0$. For small enough $\tilde\delta$ this is not
possible since $u$ is continuous and $u(x_0)>0$.

We have therefore obtained that there exists a positive $\delta_{x_0}$
such that $B_\delta(x_0)\cap G$, $\delta<\delta_{x_0}$, contains only
points of $N_i^+$.

It follows that the inequality 
\[
u(x)\leq u(x_0)
\]
is valid both for every $x\in B_\delta(x_0)\cap \partial N_i^+$ and
for every $x\in B_\delta(x_0)\cap \partial G$ (in fact
$B_\delta(x_0)\cap \partial N_i^+ = B_\delta(x_0)\cap \partial G$ as
sets). Hence each maximum point $x_0\in \partial N_i^+$ constitutes a
relative maximum of $g$ on $\partial G$. An analogous reasoning
applies for minima and relative minima on $N_j^-$ and $\partial G$,
respectively.

Since $g$ is assumed to possess only finite number of relative maxima
and minima on $\partial G$, the number of nodal domains must be
finite.
\end{proof}

Our idea in the proof of the preceding lemma has certain similarity to
that of Lemma 1.1 in Alessandrini~\cite{AlessandriniPisa} where the
number of interior critical points was considered to solution of
linear equations.

The following proof resembles the argument presented in a recent paper
by the authors, we refer to \cite{GraMar3} for more details.

\begin{proof}[Proof of Theorem~\ref{thm:uniquecont}]
  We assume for contradiction that \smallskip
  \begin{enumerate}
  \item[(A)] $u$ vanishes in a maximal open set $D\subset G$ but is
    not identically zero in $G$.
  \end{enumerate}
  The maximal open set $D$ is formed as follows: for every $x\in G$
  for which there exists an open neighborhood such that $u\equiv 0$ on
  this neighborhood we denote by $B(x,r_x)$, $r_x=\sup\left\{t>0:
    u|_{\partial B(x,t)}\equiv 0\right\}$, the maximal open
  neighborhood of $x$ where $u$ vanishes identically. Then the maximal
  open set $D$ is simply the union of all such neighborhoods. We pick
  a connected component of $D$, still denoted by $D$.

  It is worth noting that (A) implies that the boundary data function
  $g$ changes sign at least once on $\partial G$.

  By Lemma~\ref{lemma:numbernodal}, there exist positive
  $M^+,\,M^-<\infty$ such that we may index the nodal domains
  $i=1,\ldots, M^+$ and $j=1,\ldots, M^-$.

  Each nodal domain is simply-connected which can be seen as
  follows. Suppose that $N_i^+$ is not simply-connected for some
  $i\in\{1,\ldots,M^+\}$. Then there exists a Jordan curve
  $\gamma\subset N_i^+$ with its interior $S_\gamma$ and $S_\gamma$
  contains points which do not belong to the fixed nodal domain
  $N_i^+$. Moreover, $S_\gamma\subset G$ since $G$ is assumed to be
  simply-connected. It follows that the set $E=\{x\in
  S_\gamma\setminus N_i^+: u(x)\leq 0 \textrm{ or } u(x)>0\}$ is
  non-empty. If $\tilde E=\{x\in S_\gamma: u(x)<0 \textrm{ or }
  u(x)>0\}$ was empty, then $u(x)=0$ for all $x\in E$, and $u(x)>0$
  for all $x\in S_\gamma\setminus E$. This is impossible by Harnack's
  inequality, Theorem~\ref{thm:Harnack}. Hence $N_i^+$ is
  simply-connected.

  We consider next the case in which $\tilde E=\{x\in S_\gamma: u(x)<0
  \textrm{ or } u(x)>0\}\subset E$ is non-empty. It suffices to
  consider only the points at which $u<0$ (the points at which $u>0$
  are handled in the same way); this set is still denoted by $\tilde
  E$. The set $\tilde E$ is open and each component of $\tilde E$ is a
  subset of some nodal domain $N_j^-$. This contradicts with the fact
  that each nodal line meets $\partial G$. Hence $N_i^+$ is
  simply-connected.

  An analogous, symmetric, reasoning applies to $N_j^-$. Hence
  $\partial N_i^+$ and $\partial N_j^-$ are connected as the
  boundaries of simply-connected domains, and thus continua,
  i.e. compact connected sets with at least two points, for each $i$
  and $j$.

  Suppose next there exists a point $x\in\partial D\cap G$ and its
  neighborhood $B_\delta(x)$, $\delta>0$, such that
  $\overline{B}_\delta\subset G$ and $B_\delta(x)\cap\ext(D)$ contains
  only points of either $N_i^+$ or $N_j^-$ for some $i$ and $j$,
  i.e. points at which either $u>0$ or $u<0$. Assume, without loss of
  generality, that $B_\delta(x)\cap\ext(D)$ contains points of $N_i^+$
  only. Then $u\geq 0$ on $B_\delta(x)$ and by Harnack's inequality,
  Theorem~\ref{thm:Harnack}, $u\equiv 0$ on $B_{\delta/2}(x)$, which
  contradicts the maximality of the set $D$, and hence also the
  antithesis (A). In this case our claim follows.

  By the preceding reasoning it is sufficient to consider the
  following situation. For any $x\in\partial D\cap G$ and for any
  $\delta<\delta_0$, $\delta_0>0$, the neighborhood
  $B_\delta(x)\subset G$ contains points of the nodal domains $N_i^+$
  and $N_j^-$ for some indices $i$ and for some indices $j$.

  We point out that there exist a fixed index pair $(s,t)\in
  \{1,\ldots,M^+\}\times\{1,\ldots, M^-\}$ and $\delta_0>0$ such that
  each $B_\delta(x)$ contains points of $N_s^+$ and $N_t^-$, but there
  might be also points of other nodal domains in $B_\delta(x)$, for
  every $\delta<\delta_0$; this is a consequence of the fact that the
  number of nodal domains is finite in our case. We reason as follows:
  We consider a point $x\in\partial D\cap G$ and $B_\delta(x)$,
  $\delta<\delta_0$. We then select a decreasing sequence $\{\delta_i\}$
  such that $\delta_i<\delta_0$ and $\lim_{i\to\infty}\delta_i=0$. For
  each $\delta_i$ we may pick a pair of nodal domains, which we write
  \[
  a_i:=(N_{s(\delta_i)}^+, N_{t(\delta_i)}^-),
  \]
  such that $B_{\delta_i}(x)$ contains points of both nodal
  domains. Since the number of all possible nodal domain pairs is
  finite, there exists a pair which appears infinitely many times in
  the sequence $\{a_i\}$. We may hence choose this fixed pair
  $(N_s^+,N_t^-)$, where $s(\delta_{i_j})=s$ and $t(\delta_{i_j})=t$,
  for some subsequence $\{\delta_{i_j}\}$ such that
  $\lim_{j\to\infty}\delta_{i_j}=0$. It can be seen from this
  reasoning that the same pair occurs in any neighborhood
  $B_\delta(x)$, $\delta<\delta_0$.
%

  We shall next base our reasoning on some topological arguments. We
  write $\partial D_A = \{x\in\partial D: x \textrm{ is accessible
    from } D\}$,
  \[
  \partial N_{i,A}^+ = \{x\in\partial N_i^+: x \textrm{ is accessible from
  } N_i^+\},
\]
and correspondingly $\partial N_{j,A}^-$. By \cite{Newman} accessible
boundary points $\partial D_A$, $\partial N_{i,A}^+$, and $\partial
N_{j,A}^-$ are dense in $\partial D$, $\partial N_i^+$, and $\partial
N_j^-$, respectively.

We will describe a selection process which gives a pair of points
$x_1$ and $x_2$ such that $x_1,\, x_2\in \partial D_A\cap G$ and that
the associated spherical neighborhoods $B_\delta(x_1)$ and
$B_\delta(x_2)$, $\delta<\delta_0$, contain points of the same nodal
line $\partial N_s^+$, $s\in\{1,\ldots,M^+\}$ fixed. Moreover, it is
assumed that
$\overline{B}_\delta(x_1)\cap\overline{B}_\delta(x_2)=\emptyset$, and
that $\overline{B}_\delta(x_1),\,\overline{B}_\delta(x_2)\subset
G$. This procedure is as follows: We select a finite sequence of
points $\{x_l\}$, each $x_l\in\partial D_A\cap G$. As pointed out
earlier, for each $x_l$ there exists $\delta_0>0$ such that the
spherical neighborhood $B_\delta(x)$, $\delta<\delta_0$, contains
points of $N_s^+$ and $N_t^-$ for some $s\in\{1,\ldots,M^+\}$ and
$t\in\{1,\ldots, M^-\}$. Since the number of all possible nodal domain
pairs as described above is finite, after finite number of steps the
sequence $\{x_l\}$ will contain a pair of points, denoted $x_1$ and
$x_2$, which have the aforementioned properties.

We then select $x_3\in B_\delta(x_1)\cap\partial N_{s,A}^+$ and
$x_4\in B_\delta(x_2)\cap\partial N_{s,A}^+$.

We connect $x_1$ to $x_2$ by a cross-cut $\gamma_D$ in $D$, and $x_3$
to $x_4$ by a cross-cut $\gamma_{N_s^+}$ in $N_s^+$. We remark that
$x_3$, and analogously $x_4$, is accessible in $N_s^+$ with a line
segment (consult, e.g., Remark~3.3 in \cite{GraMar3}). Also $x_1$, and
analogously $x_2$, is accessible in $D$ with a line segment. We fix
such line segments to access the points $x_1,\,x_2,\,x_3$, and
$x_4$. In this way the line segments constitute part of the cross-cut
$\gamma_D$ and $\gamma_{N_s^+}$, respectively.

Since the boundary $\partial N_s^+$ is connected it is also
$\eps$-connected for every $\eps>0$. Hence for each $\eps>0$ the
points $x_1$ and $x_3$ can be joined by an $\eps$-chain
$\{a_1,\,\ldots,\,a_k\}\subset\partial N_s^+\cap G$ such that
\[
x_1 = a_1,\, a_2,\, \ldots, a_{k-1},\, a_k = x_3.
\]
We consider a collection of open balls
$\{B_{\frac{3}{2}\eps}(a_i)\}_{i=1}^k$, $a_i\in \partial N_s^+\cap G$,
such that $B_{\frac{3}{2}\eps}(a_i)\subset G$, and a domain $U_\eps^1$ which is
defined to be
\[
U_\eps^1 = \bigcup_{i=1}^kB_{\frac{3}{2}\eps}(a_i).
\]
Since $U_\eps^1$ is a domain there exists a Jordan arc,
$\gamma_{x_1x_3}^\eps$, connecting $x_1$ to $x_3$ in
$U_\eps^1$. Correspondingly, the points $x_2$ and $x_4$ can be joined
by an $\eps$-chain in $\partial N_s^+$ and we obtain a domain $U_\eps^2$
and a Jordan arc $\gamma_{x_2x_4}^\eps$ connecting $x_2$ to $x_4$ in
$U_\eps^2$.

It is worth noting that we have selected $\gamma_{x_1x_3}^\eps$ and
$\gamma_{x_2x_4}^\eps$ such that either of them does not intersect
$\gamma_D$ or $\gamma_{N_s^+}$, save the points $x_1$ and $x_2$, and
$x_3$ and $x_4$, respectively. This is possible because of the line
segment construction described above.

From the preceding Jordan arcs we obtain a Jordan curve $\Gamma^\eps$,
and by slight abuse of notation we write it as a product
\[
\Gamma^\eps=\gamma_{x_1x_3}^\eps\cdot\gamma_{N_s^+}\cdot\gamma_{x_2x_4}^\eps\cdot\gamma_D.
\]
The Jordan curve $\Gamma^\eps$ divides the plane into two disjoint
domains, and $\Gamma^\eps$ constitutes the boundary of both
domains. We consider the bounded domain, denoted by $T_\eps$, enclosed
by $\Gamma^\eps$. See Figure~\ref{Figure}.

We next deal with the Jordan domain $T_\eps$. There exists at least
one point $y\in T_\eps$ such that $u(y)<0$, i.e. $y\in N_{j_0}^-$ for
some fixed $j_0\in\{1,\ldots,M^-\}$. Assume
that this is not the case: then $u(x)\geq 0$ for every $x\in
T_\eps$. As $\gamma_D$ is part of the boundary $T_\eps$ contains also
points of $D$, and hence $u$ vanishes at such points. By Harnack's
inequality, Theorem~\ref{thm:Harnack}, $u\equiv 0$ in $T_\eps$. This
is, however, impossible since $\gamma_{N_s^+}$ is part of the boundary
of $T_\eps$, thus $u>0$ on a sufficiently small neighborhood of a
point in $\gamma_{N_s^+}$.

In an analogous way, it is possible to show that there exists a point
$z\in N_{j_0}^-\cap (G\setminus\overline{T}_\eps)$. We stress that it is
crucial that the selected points $z$ and $y$ belong to the same nodal
domain $N_{j_0}^-$. It is worth noting here that by the strong maximum
principle $\Gamma^\eps$ cannot enclose the nodal domain $N_{j_0}^-$
containing the point $y$, and therefore $G\setminus\overline{T}_\eps$
must also contain points of $N_{j_0}^-$. We then connect $z$ and $y$ in
$N_{j_0}^-$ by a Jordan arc $\gamma_{zy}$. Observe that $u(x)<0$ for every
$x\in \gamma_{zy}$.

The Jordan arc $\gamma_{zy}$ as a connected set intersects
$\Gamma^\eps$ at least at one point. We then distinguish the following
four possible cases for the point of intersection: The point of
intersection is contained in
\begin{enumerate}

\item[(1)] $\gamma_D$, 
\item[(2)] $\gamma_{N_s^+}$, 
\item[(3)] $\gamma_{x_1x_3}^\eps$, 
\item[(4)] $\gamma_{x_2x_4}^\eps$.
\end{enumerate}

In the case (1) and (2) we have reached a contradiction as $u(x)=0$
for every $x\in \gamma_D$ and $u(x)>0$ for every $x\in
\gamma_{N_s^+}$, respectively.

Consider the case (3) and case (4). We denote the point of
intersection by $x_\eps$ for every $\eps>0$. We can select an
appropriate subsequence $\{x_{\eps_j}\}_{j=1}^\infty$,
$\lim_{j\to\infty}\eps_j=0$, such that for each $j$ either
$x_{\eps_j}\in U_{\eps_j}^1$ or $x_{\eps_j}\in U_{\eps_j}^2$. We
assume, without loss of generality, that $x_{\eps_j}\in
U_{\eps_j}^1$. The sequence $\{x_{\eps_j}\}$ is clearly bounded, and
hence there exists a subsequence, still denoted $\{x_{\eps_j}\}$, such
that $\lim_{j\to\infty}x_{\eps_j} = x_0$. Observe that each
\[
x_{\eps_j}\in B_{\frac{3}{2}\eps_j}(a_m)
\]
for some $a_m\in \partial N_s^+\cap G$ in the $\eps_j$-chain. We note
that $u(a_m)=0$. Moreover, if there existed $\delta_0$ and a
subsequence, still denoted $\{x_{\eps_j}\}$, such that
\[
|u(x_{\eps_j})| \geq \delta_0 >0
\]
for every $x_{\eps_j}$, this would contradict with uniform continuity
of $u$ (note that $u$ is uniformly continuous on compact subsets of
$G$). We hence have that
\[
u(x_0)=\lim_{j\to\infty}u(x_{\eps_j})=0.
\]
In conclusion, we have reached a contradiction since $u(x_0)=0$ but,
on the other hand, $x_0\in \gamma_{zy}$ and hence $u(x_0)<0$.

All four cases (1)--(4) lead to a contradiction. Hence antithesis (A)
is false, thus the claim follows.
\end{proof}

\begin{figure}[!htbp] 
\centering
\includegraphics[width=0.8\textwidth]{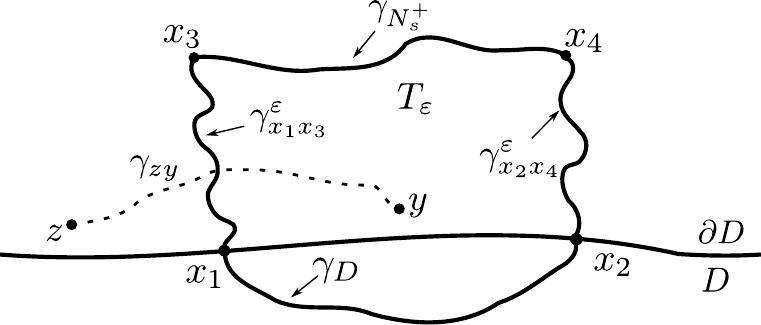}
\caption{
 Jordan domain $T_\eps$ and Jordan curve $\gamma_{zy}$ (dotted line)
 connecting $z$ to $y$ in $N_{j_0}^-$.}
\label{Figure}
 \end{figure}


Let us discuss our extra assumption in Theorem~\ref{thm:uniquecont}.

\begin{remark}
  We mention that the extra assumption, for all $x\in G$ there exists
  $r_x>0$ such that for all $r\leq r_x$ the set $\{z\in
  \overline{B}_r(x):\ u(z)=0\}$ is connected, could be replaced with
  the assumption that the set has finitely many components.
\end{remark}

\begin{remark}
  The extra assumption is closely related to the concept of
  topological monotonicity or quasi-monotonicity introduced by Whyburn
  in \cite{Whyburn}; we also refer to Astala et al.~\cite[20.1.1,
  pp. 530 ff]{Astala}.

  Let us try to clarify the role of this assumption in the proof of
  the preceding theorem. We fix there the point $x_1\in \partial
  D_A\cap G$, its neighborhood $B_\delta(x_1)$, and the point $x_3\in
  B_\delta(x_1)\cap\partial N_{s,A}^+$ (similarly $x_2\in \partial D_A\cap
  G$, $B_\delta (x_2)$, and $x_4\in B_\delta(x_2)\cap\partial
  N_{s,A}^+$). At $x_1$ and $x_3$ the function $u$ is known to
  vanish. Using the extra assumption in
  Theorem~\ref{thm:uniquecont}, we may conclude that there indeed
  exists a continuum $\mathcal{C}_\delta$ that connects $x_1$ to $x_3$
  \emph{in} $B_\delta(x_1)$ so that $u(x)=0$ for every
  $x\in\mathcal{C}_\delta$, or in other words, that the set
  $B_\delta(x_1)\cap\partial N_{s,A}^+$ is connected.

  In \cite[Remark 4.4]{GraMar3} possible spiral-like behavior that
  illustrates the role of our extra assumption is discussed in more
  detail.
\end{remark}


\end{document}